\documentclass[12pt]{article}
\usepackage{amsfonts}
\usepackage{mathrsfs}
\usepackage{fancyhdr}

\usepackage[latin1]{inputenc}
\usepackage{amsmath,amssymb}
\usepackage{latexsym}
\usepackage[active]{srcltx}
\usepackage[
bookmarks=true,         
bookmarksnumbered=true, 
colorlinks=true, pdfstartview=FitV, linkcolor=blue, citecolor=blue,
urlcolor=blue]{hyperref}

\newtheorem{theorem}{Theorem}[section]

\newtheorem{proposition}[theorem]{Proposition}
\newtheorem{lemma}[theorem]{Lemma}
\newtheorem{remark}[theorem]{Remark}
\numberwithin{equation}{section}
\def\R{{\mathbb R}}
\def\E{{{\mathbb E}\,}}

\def\N{{\mathbb N}}
\def\Var{{\mathop {{\rm Var\, }}}}
\def\cF{{\cal F}}

\def\Om{{\Omega}}

\def\square{{\vcenter{\vbox{\hrule height.3pt
        \hbox{\vrule width.3pt height5pt \kern5pt
           \vrule width.3pt}
        \hrule height.3pt}}}}

\newenvironment{proof}[1][Proof]{\noindent\textit{#1.} }{\hfill \rule{0.5em}{0.5em}}

\begin{document}

\title{\bf Second order limit laws for occupation times of the fractional Brownian motion}
\date{\empty}
\author{\ Fangjun Xu\\ \\
School of Finance and Statistics \\
East China Normal University \\
Shanghai, China 200241}

\maketitle

\begin{abstract}
\noindent We prove second order limit laws for (additive)
functionals of the $d$-dimensional fractional Brownian motion
with Hurst index $H=\frac{1}{d}$, using the method of moments, extending the Kallianpur-Robbins law.


\end{abstract}

\section{Introduction}
Let $\{B^H(t)=(B^{H,1}(t),\cdots, B^{H,d}(t)),\; t\geq 0\}$ be a $d$-dimensional fractional Brownian motion (fBm) with Hurst index $H$ in $(0,1)$. If $Hd=1$, then the local time of fractional Brownian motion $B^H$ does not exist. This is called the critical case. For example, the two-dimensional Brownian motion ($H=\frac 12$ and $d=2$) does not have local time at the origin. There is a lot of work on limit theorems for two-dimensional Brownian motion. Kallianpur and Robbins \cite{KaRo} proved that,  for any bounded and integrable function $f:\mathbb{R}^2\rightarrow \mathbb{R}$, 
\[
\frac 1{\log n} \int_0^n f(B^{\frac{1}{2}}(s))\,ds \overset{\mathcal{L}}{\longrightarrow } \frac  {Z}{2\pi} \int_{\mathbb{R}^2 }f(x)\,dx
\]
as $n$ tends to infinity, where $\overset{\mathcal{L}}{\longrightarrow }$ denotes the convergence in law and $Z$ is a random variable with exponential distribution of parameter $1$.   

A functional version of this result was given by Kasahara and Kotani in \cite{KaKo}. They also proved  the second order result. That is, when $\int_{\R^2 }f(x)\, dx=0$, 
\[
\frac{1}{\sqrt{n}}\int^{nte^{2nt}}_0 f(B^{\frac{1}{2}}(s))\, ds
\]
is weakly $M_1$-convergent to $\sqrt{\langle f\rangle}\, W(\ell(M^{-1}(t)))$ as $n$ tends to infinity, where 
\[
\langle f\rangle=-\frac{4}{\pi}\int_{\R^2}\int_{\R^2} f(x) f(y)\log|x-y|\, dx\, dy,
\] 
$W(t)$ is a one-dimensional Brownian motion, $\ell(t)$ is the local time at $0$ of another one-dimensional Brownian motion $b_1(t)$ independent of $W(t)$, and $M(t)=\max\limits_{0\leq s\leq t} b_1(s)$.  In \cite{KaKo}, Kasahara and Kotani also pointed out that
\begin{equation} \label{kkk}
\frac{1}{\sqrt{n}}\int^{e^{nt}}_0 f(B^{\frac{1}{2}}(s))\, ds \overset{f.d.}{\longrightarrow }\sqrt{\langle f\rangle}\, W(\ell(M^{-1}(t)))
\end{equation}
as $n$ tends to infinity, where $\overset{f.d.}{\longrightarrow }$ denotes the convergence of finite dimensional distributions.

The above results were extended to Markov processes, see \cite{Ka}, \cite{Kasa} and references therein. After that, the Kallianpur-Robbins law was extended to fractional Brownian motion case by K\^ono in \cite{Kon}:
\[
\frac{1}{\log t}\int^t_0 f(B^H(s))\, ds\overset{\mathcal{L}}{\longrightarrow } \Big(\frac{1}{(2\pi)^{\frac{d}{2}}}\int_{\R^d} f(x)\, dx\Big)\; Z
\] 
as $t$ tends to infinity. The corresponding functional version was obtained by Kasahara and Kosugi in  \cite{KaKos}:
\[
\frac{1}{n}\int^{e^{nt}}_0 f(B^H(s))\, ds\overset{f.d.}{\longrightarrow } \Big(\frac{1}{(2\pi)^{\frac{d}{2}}}\int_{\R^d} f(x)\, dx\Big)\; Z(t)
\]
as $n$ tends to infinity, where $Z(t)=\ell(M^{-1}(t))$. The reason of using the normalizing factor $\frac{1}{n}$ instead of $\frac{1}{\log n}$ was pointed out in Remark 1.1 of \cite{KaKos}. However, the corresponding second order limit law ($\int_{\R^d} f(x)\, dx=0$) is still open. 

In this paper, we will prove second order limit laws for the above result in \cite{KaKos}. The following two theorems are the main results of this paper. One is the limit theorem for random variables. The other is the convergence of finite dimensional distributions for stochastic processes.
\begin{theorem} \label{main} Suppose that $f$ is bounded, $\int_{\R^d} f(x)\, dx=0$ and $\int_{\R^d} |f(x)||x|^{\beta}\, dx<\infty$ for some positive constant $\beta>0$.  Then, for any $t>0$,
\[
\frac{1}{\sqrt{n}}\int^{e^{nt}}_0 f(B^H(s))\, ds \overset{\mathcal{L}}{\longrightarrow } C_{f,d}\, \sqrt{Z(t)}\, \eta
\]
as $n$ tends to infinity, where 
\[
C_{f,d}=\bigg(\frac{d\, \Gamma(\frac{d}{2})}{\pi^{\frac{d}{2}}(2\pi)^{d}} \int_{\R^d} |\widehat{f}(x)|^2 |x|^{-d}\, dx \bigg)^{\frac{1}{2}},
\]
$\widehat{f}$ is the Fourier transform of $f$, and $\eta$ is a standard normal random variable independent of $Z(t)$.
\end{theorem}

\begin{theorem} \label{main0} Suppose that $f$ is bounded, $\int_{\R^d} f(x)\, dx=0$ and $\int_{\R^d} |f(x)||x|^{\beta}\, dx<\infty$ for some positive constant $\beta>0$.  Then
\begin{align*} \label{main0eq}
\frac{1}{\sqrt{n}}\int^{e^{nt}}_0 f(B^H(s))\, ds \overset{f.d.}{\longrightarrow } C_{f,d}\, W(\ell(M^{-1}(t)))
\end{align*}
as $n$ tends to infinity, where $W(t)$ is a one-dimensional Brownian motion independent of $B^H(\cdot)$ and $\ell(M^{-1}(t))$.
\end{theorem}

\begin{remark} Since the process $\ell(M^{-1}(t))$ is not in $C((0,\infty])$, we could not use the Skorohod $J_1$-topology. For properties of the process $Z(t)=\ell(M^{-1}(t))$, we refer to \cite{Ka} and references therein. So far, we still have no idea to show the weakly $M_1$-convergence of the result in Theorem \ref{main0}, which was proved to be true for two-dimensional Brownian motion. 
\end{remark}  

\begin{remark} Since the function $f$ in the above two theorems is bounded, the constant $\beta$ in Theorems \ref{main} and \ref{main0} can always be assumed to be less than or equal to 1.
\end{remark}

As we all know, the fractional Brownian motion with Hurst index not equal to $\frac{1}{2}$ is neither a Markov process nor a semimartingale. Therefore, the methods once applied for two-dimensional Brownian motion and Markov processes can not be used here to prove Theorems \ref{main} and \ref{main0}. Moreover, when proving limit theorems for (additive) functionals of fractional Brownian motion, one often uses the method of moments. Another possible candidate is the Malliavin calculus. 

To show Theorems \ref{main} and \ref{main0}, we would use the method of moments plus some kind of chaining argument. The chaining argument was first developed in \cite{nx1} to prove the central limit theorem for an additive functional of the $d$-dimensional fractional Brownian motion with Hurst index $H\in(\frac{1}{d+2},\frac{1}{d})$. It has been proved to be very powerful when obtaining the asymptotic behavior of moments. However, the situation here is a little different from that in \cite{nx1}. We consider fractional Brownian motions in the critical case and use a different normalizing factor. To use the chaining argument in \cite{nx1}, some modifications and new ideas are needed. 

The main difficulty when applying the method of moments comes from the convergence of even moments. To overcome it, we first estimate the covariance between two increments of the fractional Brownian motion with Hurst index $H<\frac{1}{2}$ and then show that, under some conditions, these covariances do not contribute to the limit of even moments. See Lemma \ref{cov} and Step 3 in the proof of Proposition \ref{main1} for details. Roughly speaking, the short range dependence for the fractional Brownian motion with Hurst index $H<\frac{1}{2}$ and proper normalizing factor enable us to have some kind of weak independence. When showing Theorem \ref{main0}, we would use the above mentioned techniques and the known result for two-dimensional Brownian motion in (\ref{kkk}).

The paper is outlined in the following way. After some preliminaries in Section 2,  Sections 3 and 4 are devoted to the
proof of Theorems \ref{main} and \ref{main0}, respectively. Throughout this paper, if not mentioned otherwise, the letter $c$, 
with or without a subscript, denotes a generic positive finite
constant whose exact value is independent of $n$ and may change from
line to line. Moreover, we use $\iota$ to denote $\sqrt{-1}$ and $(x,y)$ (or $x\cdot y$) to denote the usual inner product in $\R^d$.

\bigskip

\section{Preliminaries}

Let $\{B^H(t)=(B^{H,1}(t),\cdots, B^{H,d}(t)),\; t\geq 0\}$ be a $d$-dimensional fractional Brownian motion (fBm) with Hurst index $H$ in $(0,1)$, defined on some probability space $(\Om, \cF, P)$. That is, the components
of $B^H$ are independent centered Gaussian processes with covariance function
\[
\E\big(B^{H,i}(t)B^{H,i}(s)\big)=\frac{1}{2}\big(t^{2H}+s^{2H}-|t-s|^{2H}\big).
\]  
The following lemma gives comparable upper and lower bounds for increments of the $d$-dimensional fractional Brownian motion on $n$ adjacent intervals. 
\begin{lemma} \label{lnd}
Given $n\ge 1$, there exist two positive constants $\kappa_{1}$ and $\kappa_{2}$ depending only on  $n$, $H$ and $d$, such that for any $0=s_0<s_1 <\cdots < s_n $ and $x_i \in \mathbb{R}^d$, $1\le i \le n$, we have
\[
\kappa_{1} \sum_{i=1}^n |x_i|^2 (s_i -s_{i-1})^{2H} \le \mathrm{Var} \Big( \sum_{i=1}^n x_i  \cdot \big(B^H(s_i) -B^H(s_{i-1})\big) \Big) \le \kappa_{2}  \sum_{i=1}^n |x_i|^2 (s_i -s_{i-1})^{2H}.
\]
\end{lemma} 
\begin{proof} The second inequality is obvious. So it suffices to show the first one, which follows directly from the local nondeterminism property of the fractional Brownian motion. 
\end{proof}

The above inequalities in Lemma \ref{lnd} can be reformulated as
\begin{equation} \label{eq1}
\kappa_{1} \sum_{i=1}^n \Big|\sum_{j=i}^n x_j \Big|^2  (s_i -s_{i-1})^{2H} \le \mathrm{Var} \Big( \sum_{i=1}^n x_i  \cdot B^H_{s_i}  \Big) \le \kappa_{2}  \sum_{i=1}^n \Big|  \sum_{j=i}^n x_j \Big|^2(s_i -s_{i-1})^{2H}.
\end{equation}

The next lemma gives a formula for moments of $\sqrt{Z(t)}\, \eta$ where $Z(t)$ is an exponential random variable with parameter $t$ and $\eta$ is a standard normal random variable independent of $Z(t)$.
\begin{lemma} \label{m} For any $m\in\N$ and $t>0$, 
\[
\E[\sqrt{Z(t)}\,\eta]^{2m}=\frac{(2m)!\, t^m}{2^m}.
\]
\end{lemma}
\begin{proof} Using the moment generating function of the exponential distribution, we can easily obtain $\E[Z(t)]^m=m!\, t^m$.  Since $\eta$ and $Z(t)$ are independent, 
\[
\E[\sqrt{Z(t)}\,\eta]^{2m}=\E[Z(t)]^m\E[\eta]^{2m}=m!\, t^m\, (2m-1)!!=\frac{(2m)!\, t^m}{2^m}.
\]
\end{proof}
\begin{remark} The above result says that $\sqrt{Z(t)}\,\eta$ has Laplace distribution.
\end{remark}

We shall also need the following lemma which will play a very important role in proving the convergence of finite dimensional distributions. 
\begin{lemma} \label{cov} For any $H<\frac{1}{2}$ and $0<t_1<t_2<t_3<t_4<\infty$, we have
\[
\big|\E\big(B^{H,1}(t_4)-B^{H,1}(t_3)\big)\big(B^{H,1}(t_2)-B^{H,1}(t_1)\big)\big|
\]
\begin{itemize}
\item[(i)]
is less than 
\[
2H\Big(\frac{\Delta t_2}{\Delta t_3}\Big)^{\frac{1}{2}-H}\Big(\frac{\Delta t_4}{\Delta t_3}\Big)^{\frac{1}{2}-H}(\Delta t_2)^H(\Delta t_4)^H,
\]
where $\Delta t_i=t_i-t_{i-1}$ for $i=2,3,4$;
\item[(ii)] is less than
\[
2\Big(\frac{\Delta t_2\wedge \Delta t_4}{\Delta t_2\vee \Delta t_4}\Big)^{H}(\Delta t_2)^H(\Delta t_4)^H.
\]
\end{itemize}
\end{lemma}
\begin{proof} Part (ii) was pointed out by  K\^ono in \cite{Kon}. So it suffices to show part (i). For simplicity of notation, we write $a$ for 
\[
\E\big(B^{H,1}(t_4)-B^{H,1}(t_3)\big)\big(B^{H,1}(t_2)-B^{H,1}(t_1)\big).
\]
Since $H<\frac{1}{2}$, it is easy to see
\begin{align*}
|a|
&=(\Delta t_4+\Delta t_3)^{2H}+(\Delta t_3+\Delta t_2)^{2H}-(\Delta t_4+\Delta t_3+\Delta t_2)^{2H}-(\Delta t_3)^{2H}\\
&=(\Delta t_3)^{2H}\big[(1+u)^{2H}+(1+v)^{2H}-(1+u+v)^{2H}-1\big],
\end{align*}
where $u=\frac{\Delta t_2}{\Delta t_3}$ and $v=\frac{\Delta t_4}{\Delta t_3}$.

Using Taylor's expansion and the fact $H<\frac{1}{2}$,  we can show that 
\[
\big[(1+u)^{2H}+(1+v)^{2H}-(1+u+v)^{2H}-1\big]
\]
is less than $2Hu$ and $2Hv$. Therefore,
\begin{align*}
|a|\leq 2H (\Delta t_3)^{2H}\sqrt{uv}=2H (uv)^{\frac{1}{2}-H}(\Delta t_2)^H(\Delta t_4)^H.
\end{align*}
This completes the proof.
\end{proof}

\bigskip

\section{Proof of Theorem \ref{main}}
In this section, we shall show Theorem \ref{main}. Since $f$ is bounded,  we only need to consider the convergence of the following random variables
\[
F_n=\frac{1}{\sqrt{n}}\int^{e^{nt}}_1 f(B^H(s))\, ds.
\]

For any $m\in\N$, let
\[
I^n_m= \frac{m!}{n^{\frac{m}{2}}}\, \E\Big[\int_{D_{m,1}} \Big(\prod^m_{i=1}f(B^H(s_i))\Big)\, ds\Big],
\]
where $D_{m,1}=\big\{(s_1, \dots, s_m) \in D_m: s_i-s_{i-1}\geq n^{-m},\, i=2,3,\dots, m\big\}$ and $D_m=\big\{1<s_1<\cdots<s_m<e^{nt} \big\}$.
Then,   taking into account that $f$ is bounded, we can obtain
\begin{align*}
\Big|\E(F_n)^m-I^n_m\Big|  & \le \frac{m!}{n^{\frac{m}{2}}}\, \sum_{j=1}^m \E\Big[\int_{D_{m}\cap\{|s_j -s_{j-1}| <n^{-m}\}} \Big(\prod^m_{i=1}|f(B^H(s_i))|\Big)\, ds\Big]\\
& \le \|f\|_\infty\, \frac{m m!}{n^{\frac{3m}{2}}}\,  \E\Big[\int_{D_{m-1}} \Big(\prod^{m-1}_{i=1}|f(B^H(s_i))|\Big)\, ds\Big].
\end{align*}
Thus,  using the argument on page 166 of \cite{KaKos}, 
\begin{equation} \label{eq1}
  \Big|\E(F_n)^m-I^n_m\Big|\le c_1 n^{-\frac m2 -1}.
\end{equation}
Applying Fourier transform, we can write
\begin{align*}
I^n_m
&=\frac{m!}{((2\pi)^d\sqrt{n})^m}\int_{\R^{md}}\int_{D_{m,1}} \Big(\prod^m_{i=1}\widehat{f}(x_i)\Big)\,\exp\Big(-\frac{1}{2}\Var\big(\sum\limits^m_{i=1} x_i\cdot B^H(s_i)\big)\Big)\, ds\, dx.
\end{align*}
Making the change of variables $y_i=\sum\limits^m_{j=i}x_j$ for $i=1,2,\dots, m$ gives
\begin{align*}
I^n_m
&=\frac{m!}{((2\pi)^d\sqrt{n})^m}\int_{\R^{md}}\int_{D_{m,1}} \Big(\prod^m_{i=1}\widehat{f}(y_i-y_{i+1})\Big)\\
&\qquad\qquad\qquad\times \exp\bigg(-\frac{1}{2}\Var\Big(\sum\limits^m_{i=1} y_i\cdot \big(B^H(s_i)-B^H(s_{i-1})\big)\Big)\bigg)\, ds\, dy.
\end{align*}

Set $I^n_{m,0}=I^n_m$. For $k=1,\dots,m$, we define
\begin{align*}
I^n_{m,k}
&=\frac{m!}{((2\pi)^d\sqrt{n})^m}\int_{\R^m}\int_{D_{m,1}}  I_k\,  \prod^{m}_{i=k+1} \widehat{f}(y_i-y_{i+1})\\
&\qquad\qquad\times\exp\bigg(-\frac{1}{2}\Var\Big(\sum\limits^m_{i=1} y_i\cdot \big(B^H(s_i)-B^H(s_{i-1})\big)\Big)\bigg)\, ds\, dy,
\end{align*}
where 
\[
I_k =
\begin{cases}
\prod\limits^{\frac{k-1}{2}}_{j=1}|\widehat{f}(y_{2j})|^2 \widehat{f}(-y_{k+1}), & \text{if } k \text{ is odd}; \\
\prod\limits^{\frac{k}{2}}_{j=1}|\widehat{f}(y_{2j})|^2, & \text{if }k\text{ is even}.
\end{cases}
\]

The following proposition, which is similar to Proposition 4.1 in \cite{nx1}, controls the difference between $I^n_{m,k-1}$ and $ I^n_{m,k}$. We fix a positive constant $\lambda$ strictly less than $\frac{1}{2}$.

\begin{proposition} \label{chain} 
For $k=1,2,\dots,m$, there exists a positive constant $c$, which depends only on $\lambda$, such that
\[
|I^n_{m,k-1}-I^n_{m,k}|\leq c\, n^{-\lambda}.
\]
\end{proposition}
 
\begin{proof}  We first consider the case when $k$ is odd. Making the change of variables $u_1=s_1$, $u_i =s_i -s_{i-1}$, for $2\le i\le m$ and then applying Lemma \ref{lnd}, we can show that $|I^n_{m,k-1}-I^n_{m,k}|$ is less than a constant multiple of
\begin{align*}
&n^{-\frac{m}{2}}\int_{\R^{md}}\int_{ [n^{-m}, e^{nt}]^m} \Big(\prod^m_{i=k+1}  |\widehat{f}(y_i-y_{i+1})|\Big)\, \big|\widehat{f}(y_k-y_{k+1})-\widehat{f}(-y_{k+1})\big| \\
&\qquad\qquad\qquad \times \Big(\prod^{\frac{k-1}{2}}_{j=1}|\widehat{f}(y_{2j})|^2\Big)\,\exp\Big(-\frac{\kappa_1}{2}\sum\limits^m_{i=1} |y_i|^2\, u^{2H}_i\Big)\, du\, dy
\end{align*}
with the convention $y_{m+1}=0$.

Taking into account that $|\widehat{f}(x)|\leq c_{\alpha}(|x|^{\alpha}\wedge 1)$ for any $\alpha\in[0,\beta]$,
$|I^n_{m,k-1}-I^n_{m,k}|$ is less than a constant multiple of
\[
n^{-\frac{m}{2}}\int_{\R^{md}}\int_{ [n^{-m}, e^{nt}]^m}  |y_{k}|^{\alpha} \prod^{\lfloor\frac{m}{2}\rfloor}_{j=\frac{k+1}{2}}  (|y_{2j}|^{\alpha}+|y_{2j+1}|^{\alpha}) \Big(\prod^{\frac{k-1}{2}}_{j=1}|\widehat{f}(y_{2j})|^2\Big) \exp\Big(-\frac{\kappa_1}{2}\sum\limits^m_{i=1} |y_i|^2\, u^{2H}_i\Big)\, du\, dy.
\]
Integrating with respect to $y_i$s and $u_i$s with $i\leq k-1$ gives
\begin{align*}
|I^n_{m,k-1}-I^n_{m,k}|
&\leq c_1\, n^{-\frac{m-(k-1)}{2}}\int_{\R^{(m-k+1)d}}\int_{[n^{-m}, e^{nt}]^{m-k+1}}  |y_k|^{\alpha}\prod^{\lfloor\frac{m}{2}\rfloor}_{j=\frac{k+1}{2}}  (|y_{2j}|^{\alpha}+|y_{2j+1}|^{\alpha})\\
&\qquad\qquad\times \exp\Big(-\frac{\kappa_1}{2}\sum\limits^m_{i=k} |y_i|^2\, u^{2H}_i\Big)\, d\overline{u}\, d\overline{y},
\end{align*}
where $d\overline{u}=du_k\cdots du_{m}$ and $d\overline{y}=dy_k\cdots dy_{m}$. After doing some calculus,
\begin{align*}
|I^n_{m,k-1}-I^n_{m,k}|
&\leq c_2\, n^{-\frac{m-k+1}{2}+(\lfloor\frac{m-k+1}{2}\rfloor+1)(mH\alpha)+(m-k-\lfloor\frac{m-k+1}{2}\rfloor)}\\
&= c_2\, n^{\frac{m}{2}-\lfloor\frac{m}{2}\rfloor-1+(\lfloor\frac{m-k+1}{2}\rfloor+1)(mH\alpha)}.
\end{align*}
Choosing $\alpha$ small enough such that 
\[
\frac{m}{2}-\lfloor\frac{m}{2}\rfloor-1+(\lfloor\frac{m-k+1}{2}\rfloor+1)(mH\alpha)=-\lambda
\] 
gives
\begin{equation} \label{chaino}
|I^n_{m,k-1}-I^n_{m,k}|\leq c_2\, n^{-\lambda}.
\end{equation}

We next consider the case when $k$ is even. By Lemma \ref{lnd}, $|I^n_{m,k-1}-I^n_{m,k}|$ is less than a constant multiple of
\begin{align*}
&n^{-\frac{m}{2}}\int_{\R^{md}}\int_{[n^{-m}, e^{nt}]^m} \big|\widehat{f}(-y_k)\big|\big|\widehat{f}(y_k-y_{k+1})-\widehat{f}(y_k)\big|\Big(\prod^m_{i=k+1}  |\widehat{f}(y_i-y_{i+1})|\Big) \\
&\qquad\qquad\qquad \times \Big(\prod^{\frac{k-2}{2}}_{j=1}|\widehat{f}(y_{2j})|^2\Big)\,\exp\Big(-\frac{\kappa_1}{2}\sum\limits^m_{i=1} |y_i|^2\, u^{2H}_i\Big)\, du\, dy.
\end{align*}
Using similar arguments as in the odd case,
\begin{align*}
&|I^n_{m,k-1}-I^n_{m,k}|\\
&\leq c_3\, n^{-\frac{m}{2}}\int_{\R^{md}}\int_{[n^{-m}, e^{nt}]^m} |y_k|^{\alpha}|y_{k+1}|^{\alpha} \prod^{\lfloor\frac{m}{2}\rfloor}_{j=\frac{k+2}{2}}  (|y_{2j}|^{\alpha}+|y_{2j+1}|^{\alpha}) \\
&\qquad\qquad\qquad \times \Big(\prod^{\frac{k-2}{2}}_{j=1}|\widehat{f}(y_{2j})|^2\Big)\,\exp\Big(-\frac{\kappa_1}{2}\sum\limits^m_{i=1} |y_i|^2\, u^{2H}_i\Big)\, du\, dy\\
&\leq c_4\, n^{-\frac{m-(k-2)}{2}+(\lfloor\frac{m-k}{2}\rfloor+2)(mH\alpha)+(m-k-\lfloor\frac{m-k}{2}\rfloor)}\\
&= c_4\, n^{\frac{m}{2}-\lfloor\frac{m}{2}\rfloor-1+(\lfloor\frac{m-k}{2}\rfloor+2)(mH\alpha)}.
\end{align*}
Choosing $\alpha$ small enough such that 
\[
\frac{m}{2}-\lfloor\frac{m}{2}\rfloor-1+(\lfloor\frac{m-k}{2}\rfloor+2)(mH\alpha)=-\lambda
\] 
gives
\begin{equation} \label{chaine}
|I^n_{m,k-1}-I^n_{m,k}|\leq c_4\, n^{-\lambda}.
\end{equation}

Combining (\ref{chaino}) and (\ref{chaine}) gives the desired estimates.
\end{proof}
\bigskip

In the sequel, we will use estimates in Proposition \ref{chain} to show Theorem \ref{main}.
\begin{proposition} \label{main1} Suppose that $f$ is bounded, $\int_{\R^d} f(x)\, dx=0$ and $\int_{\R^d} |f(x)||x|^{\beta}\, dx<\infty$ for some positive constant $\beta>0$. Then, for any $t>0$,
\[
\frac{1}{\sqrt{n}}\int^{e^{nt}}_1 f(B^H(s)) ds \overset{\mathcal{L}}{\longrightarrow } C_{f,d}\, \sqrt{Z(t)}\, \eta
\]
as $n$ tends to infinity, where 
\[
C_{f,d}=\bigg(\frac{d\, \Gamma(\frac{d}{2})}{\pi^{\frac{d}{2}}(2\pi)^{d}} \int_{\R^d} |\widehat{f}(x)|^2 |x|^{-d}\, dx \bigg)^{\frac{1}{2}},
\]
$\widehat{f}$ is the Fourier transform of $f$, $Z(t)$ is an exponential random variable with parameter $t$ and $\eta$ is a standard normal random variable independent of $Z(t)$.
\end{proposition}

\begin{proof}  The proof will be done in several steps.

\medskip \noindent
\textbf{Step 1} \quad  We first show tightness.  Let $F_n=\frac{1}{\sqrt{n}}\int^{e^{nt}}_1 f(B^H(s))\, ds$. Using Fourier transform,
\begin{align*}
\E(F_n)^2
&=\frac{2}{(2\pi)^{2d}n}\int^{e^{nt}}_1\int^{s_2}_1\int_{\R^{2d}} \widehat{f}(x_1)\widehat{f}(x_2)\, \exp\Big(-\frac{1}{2}\Var\big(\sum^2_{i=1} x_i\cdot B^H(s_i)\big)\Big)\, ds\, dx.
\end{align*}
Since $|\widehat{f}(x)|\leq c_{\alpha} (|x|^{\alpha}\wedge 1)$ for all $x\in\R^d$ and any $\alpha\in[0,\beta]$, by Lemma \ref{lnd},
\begin{align*}
\E(F_n)^2
&\leq \frac{c_1}{n}\int^{e^{nt}}_1\int^{s_2}_1\int_{\R^{2d}} |\widehat{f}(x_2)|\, \exp\Big(-\frac{\kappa_1}{2} |x_2|^2(s_2-s_1)^{2H}-\frac{\kappa_1}{2} |x_2+x_1|^2s^{2H}_1\Big)\, ds\, dx\\
&\leq \frac{c_2}{n}\Big(\int^{e^{nt}}_1s^{-1}_1\, ds_1\Big)\Big(\int_{\R^d} |\widehat{f}(x_2)||x_2|^{-d}\, dx_2\Big)\\
&\leq c_3\, t.
\end{align*}

\medskip \noindent
\textbf{Step 2} \quad We show the convergence of odd moments.  Assume that $m$ is odd. Recall the estimate  (\ref{eq1}), which allows us to replace 
$\E(F_n)^{m}$ by $I^n_{m}$. 
By Proposition \ref{chain}, we only need to show 
\[
\lim_{n\to\infty} I^n_{m,m}=0,
\]
where
\begin{align*}
I^n_{m, m}
&=\frac{m!}{((2\pi)^d \sqrt{n})^m}\int_{\R^{md}}\int_{D_{m,1}}  \widehat{f}(y_m)  \prod^{\frac{m-1}{2}}_{j=1} |\widehat{f}(y_{2j})|^2\\
&\qquad\qquad \times \exp\bigg(-\frac{1}{2}\Var\Big(\sum\limits^m_{i=1} y_i\cdot \big(B^H(s_i)-B^H(s_{i-1})\big)\Big)\bigg)\, ds\, dy.
\end{align*}

We make the change of variables $u_1=s_1$, $u_i =s_i -s_{i-1}$ for $2\le i\le m$. By Lemma \ref{lnd},
\[
|I^n_{m, m}|\leq \frac{m!}{((2\pi)^d  \sqrt{n})^m}\int_{\R^{md}}\int_{O_{m}}  |\widehat{f}(y_m)|  \prod^{\frac{m-1}{2}}_{j=1} |\widehat{f}(y_{2j})|^2\, \exp\Big(-\frac{\kappa_1}{2}\sum\limits^m_{i=1} |y_i|^2\, u^{2H}_i\Big)\, du\, dy,
\]
where, as before,
\[
O_{m}=\big\{(u_1, \dots, u_{m}): \, 1<u_1, \sum_{i=1}^m u_i <e^{nt}, n^{-m}<u_i<e^{nt},\, i=2,\dots, m \big\}.
\]
Note that $O_m \subset [1,e^{nt}] \times [n^{-m}, e^{nt}]^{m-1}$. As a consequence,
\begin{align*}
\big| I^n_{m, m}\big|
&\leq c_4\, n^{-\frac{m}{2}}\int_{\R^{md}}\int_{[1,e^{nt}] \times [n^{-m}, e^{nt}]^{m-1}}  |\widehat{f}(y_m)|\prod^{\frac{m-1}{2}}_{j=1} |\widehat{f}(y_{2j})|^2\, \exp\Big(-\frac{\kappa_1}{2}\sum\limits^m_{i=1} |y_i|^2\, u^{2H}_i\Big)\, du\, dy\\
&\leq c_5\,n^{-\frac{m}{2}}\, \Big(\int_{\R^d} |\widehat{f}(y)|^2|y|^{-d}dy\Big)^{\frac{m-1}{2}} \Big(\int^{e^{nt}}_{n^{-m}} u^{-1}\, du\Big)^{\frac{m-1}{2}}\, \Big(\int_{\R^d}|\widehat{f}(y)|\, |y|^{-d} dy\Big)\\
&\leq c_6\, n^{-\frac{1}{2}}.
\end{align*}
Combining these estimates gives $\lim\limits_{n\to\infty}\E(F_n)^{m}=0$ when $m$ is odd.

\medskip \noindent
\textbf{Step 3} \quad We show the convergence of even moments. Assume that $m$ is even. Recall  the estimate (\ref{eq1}). 
By Proposition \ref{chain}, it suffices to show 
\begin{align} \label{2nd1-0}
\lim_{n\to\infty} I^n_{m,m}=\Big( \frac{d\, \Gamma(\frac{d}{2})}{\pi^{\frac{d}{2}}(2\pi)^{d}}\int_{\R^d}|\widehat{f}(x)|^2|x|^{-d}\, dx\Big)^{\frac{m}{2}}\, \E\Big(\sqrt{Z(t)}\, \eta\Big)^{m},
\end{align}
where
\begin{align} \nonumber \label{inmm}
I^n_{m, m}
&=\frac{m!}{((2\pi)^d  \sqrt{n})^m}\int_{\R^{md}}\int_{D_{m,1}} \Big(\prod^{m/2}_{j=1} |\widehat{f}(y_{2j})|^2\Big)\\ 
&\qquad\qquad\qquad\times\exp\bigg(-\frac{1}{2}\Var\Big(\sum\limits^m_{i=1} y_i\cdot \big(B^H(s_i)-B^H(s_{i-1})\big)\Big)\bigg)\, ds\, dy.
\end{align}

For $i=1,2,\dots, m$, set $\Delta s_i=s_i-s_{i-1}$ with the convention $s_0=0$.  Define 
\begin{align*}
\widehat{I}^n_{m, m}
&= \frac{m!}{((2\pi)^d  \sqrt{n})^m}\int_{\R^{md}}\int_{\widehat{O}_m} \Big(\prod^{m/2}_{j=1} |\widehat{f}(y_{2j})|^2\Big)\\
&\qquad\qquad\qquad\times\exp\bigg(-\frac{1}{2}\Var\Big(\sum\limits^m_{i=1} y_i\cdot \big(B^H(s_i)-B^H(s_{i-1})\big)\Big)\bigg)\, ds\, dy,
\end{align*}
where
\[
\widehat{O}_m=\big\{(s_1,s_2,\dots, s_m): 1<\Delta s_1<\infty,\, n^{-m}<\Delta s_i<e^{nt},\, i=2,\dots,m\big\}.
\]

Note that $D_{m,1}\subseteq \widehat{O}_m$. We obtain $I^n_{m, m}\leq \widehat{I}^n_{m, m}$. Let
\begin{align*}
\widetilde{I}^n_{m, m}
&= \frac{m!}{((2\pi)^d  \sqrt{n})^m}\int_{\R^{md}}\int_{\widetilde{O}_m} \Big(\prod^{m/2}_{j=1} |\widehat{f}(y_{2j})|^2\Big)\\
&\qquad\qquad\qquad\times\exp\bigg(-\frac{1}{2}\Var\Big(\sum\limits^m_{i=1} y_i\cdot \big(B^H(s_i)-B^H(s_{i-1})\big)\Big)\bigg)\, ds\, dy,
\end{align*}
where
\[
\widetilde{O}_m=\big\{(s_1,s_2,\dots,s_m):  n^2<\Delta s_{2i-1}<e^{nt},\, n^{-1}<\Delta s_{2i}<n,\, i=1,2,\dots,m/2 \big\}.
\]
Then, by Lemma \ref{lnd}, 
\begin{align} \label{prop1}
\lim\sup_{n\to\infty}I^n_{m, m}=\lim\sup_{n\to\infty}\widetilde{I}^n_{m, m}.
\end{align}

For any constant $\gamma>1$, we define 
\begin{align*}
\widetilde{I}^{n,\gamma}_{m, m}
&= \frac{m!}{((2\pi)^d  \sqrt{n})^m}\int_{\R^{md}}\int_{\widetilde{O}_{m,\gamma}} \Big(\prod^{m/2}_{j=1} |\widehat{f}(y_{2j})|^2\Big)\\
&\qquad\qquad\qquad\times\exp\bigg(-\frac{1}{2}\Var\Big(\sum\limits^m_{i=1} y_i\cdot \big(B^H(s_i)-B^H(s_{i-1})\big)\Big)\bigg)\, ds\, dy,
\end{align*}
where
\[
\widetilde{O}_{m,\gamma}=\widetilde{O}_m\cap \Big\{\frac{\Delta s_{2i-1}}{\Delta s_{2j-1}}>\gamma\; \text{or }\; \frac{\Delta s_{2j-1}}{\Delta s_{2i-1}}>\gamma\quad \text{for all}\; i,j=1,2,\dots,m/2 \Big\}.
\] 

By Lemma \ref{lnd}, $|\widetilde{I}^n_{m, m}-\widetilde{I}^{n,\gamma}_{m, m}|$ is less than a constant multiple of $\frac{\log \gamma}{n}$.
Note that
\begin{align*}
\Var\Big(\sum\limits^m_{i=1} y_i\cdot \big(B^H(s_i)-B^H(s_{i-1})\big)\Big)=\sum\limits^m_{i=1} |y_i|^2(\Delta s_i)^{2H}+2\sum_{i<j}(y_i,y_j)\, a_{ij},
\end{align*}
where $a_{ij}=-\frac{1}{2}\big[(s_j-s_i)^{2H}+(s_{j-1}-s_{i-1})^{2H}-(s_j-s_{i-1})^{2H}-(s_{j-1}-s_i)^{2H}\big]$.

Recall $H=\frac{1}{d}\leq\frac{1}{2}$. On the domain $\widetilde{O}_{m,\gamma}$, by Lemma \ref{cov}, we have the following estimates:
\begin{itemize}
\item[(i)] If both $i$ and $j$ are even, then 
\[
|a_{ij}|\leq c_7(1-2H)\frac{(\Delta s_i)^H(\Delta s_j)^H}{n^{(1-2H)}};
\]
\item[(ii)] If both $i$ and $j$ are odd, then
\begin{align*}
|a_{ij}|&\leq c_7(1-2H)\Big(\frac{\Delta s_i\wedge \Delta s_j}{\Delta s_i\vee \Delta s_j}\Big)^H(\Delta s_i)^H(\Delta s_j)^H\\
&\leq c_7(1-2H)\frac{(\Delta s_i)^H(\Delta s_j)^H}{\gamma^H};
\end{align*}
\item[(iii)] If one and only one of $i$ and $j$ is odd, then  
\begin{align*}
|a_{ij}|&\leq c_7(1-2H)\Big(\frac{\Delta s_i\wedge \Delta s_j}{\Delta s_i\vee \Delta s_j}\Big)^H(\Delta s_i)^H(\Delta s_j)^H\\
&\leq c_7(1-2H)\frac{(\Delta s_i)^H(\Delta s_j)^H}{n^{H}}.
\end{align*}
\end{itemize}

Let $\theta=(1-2H)\wedge H$. Applying the Cauchy-Schwartz inequality and the above estimates, we see that
\[
\Var\Big(\sum\limits^m_{i=1} y_i\cdot \big(B^H(s_i)-B^H(s_{i-1})\big)\Big)
\]
is between 
\[
\big(1-\frac{c_8}{\gamma^{H}}-\frac{c_9}{n^{\theta}}\big)\sum\limits^m_{i=1} |y_i|^2(\Delta s_i)^{2H}
\]
and
\[
\big(1+\frac{c_8}{\gamma^{H}}+\frac{c_9}{n^{\theta}}\big)\sum\limits^m_{i=1} |y_i|^2(\Delta s_i)^{2H}.
\]
Therefore,
\begin{align*}
\widetilde{I}^{n,\gamma}_{m, m}
&\leq \frac{m!}{((2\pi)^d  \sqrt{n})^m}\int_{\R^{md}}\int_{\widetilde{O}_{m,\gamma}} \Big(\prod^{m/2}_{j=1} |\widehat{f}(y_{2j})|^2\Big)\\
&\qquad\qquad\qquad\times\exp\bigg(-\frac{1}{2}\big(1-\frac{c_8}{\gamma^{H}}-\frac{c_9}{n^{\theta}}\big)\sum\limits^m_{i=1} |y_i|^2(\Delta s_i)^{2H}\bigg)\, ds\, dy\\
&\leq \frac{m!}{((2\pi)^d  \sqrt{n})^m}\int_{\R^{md}}\int_{\widetilde{O}_m} \Big(\prod^{m/2}_{j=1} |\widehat{f}(y_{2j})|^2\Big)\\
&\qquad\qquad\qquad\times\exp\bigg(-\frac{1}{2}\big(1-\frac{c_8}{\gamma^{H}}-\frac{c_9}{n^{\theta}}\big)\sum\limits^m_{i=1} |y_i|^2(\Delta s_i)^{2H}\bigg)\, ds\, dy.
\end{align*}

After doing some calculus, we can obtain
\begin{align*}
\limsup_{n\to\infty}\widetilde{I}^{n,\gamma}_{m, m}
&\leq m!\, t^{\frac{m}{2}}\, (1-\frac{c_8}{\gamma^{H}})^{-\frac{md}{2}}\,\bigg(\frac{1}{(2\pi)^{\frac{d}{2}}} \int^{\infty}_0e^{-\frac{u^{2H}}{2}}\, du\bigg)^{\frac{m}{2}}\, \bigg(\frac{1}{(2\pi)^d}\int_{\R^{d}} |\widehat{f}(x)|^2|x|^{-d}\, dx\bigg)^{\frac{m}{2}}.
\end{align*}
Note that 
\[
\limsup_{n\to\infty}\widetilde{I}^{n}_{m, m}\leq \limsup_{n\to\infty}\widetilde{I}^{n,\gamma}_{m, m}+\limsup_{n\to\infty}|\widetilde{I}^{n}_{m, m}-\widetilde{I}^{n,\gamma}_{m, m}|.
\]
Combining the above two inequalities with (\ref{prop1}) gives 
\begin{align} \label{sup}
\limsup_{n\to\infty}I^{n}_{m, m}\leq m!\, t^{\frac{m}{2}}\, \bigg(\frac{1}{(2\pi)^{\frac{d}{2}}}\int^{\infty}_0e^{-\frac{u^{2H}}{2}}\, du\bigg)^{\frac{m}{2}}\, \bigg(\frac{1}{(2\pi)^d}\int_{\R^{d}} |\widehat{f}(x)|^2|x|^{-d}\, dx\bigg)^{\frac{m}{2}}.
\end{align}

Recall the definition of $I^n_{m, m}$ in (\ref{inmm}). Integrating with respect to $y_{m-1}$ and then using the following inequality
\begin{equation} \label{lower}
\int_{\R^d} e^{-\frac{1}{2}|x_1|^2u^{2H}-v x_1\cdot x_2}\, dx_1\geq (2\pi)^{\frac{d}{2}} u^{-1},
\end{equation}
\begin{align*}
I^n_{m, m}
&\geq \frac{m!}{(2\pi)^{\frac{(2m-1)d}{2}}n^{\frac{m}{2}}}\int_{\R^{(m-1)d}}\int_{D_{m,1}} \Big(\prod^{m/2}_{j=1} |\widehat{f}(y_{2j})|^2\Big)(\Delta s_{m-1})^{-1}\\
&\qquad\times\exp\bigg(-\frac{1}{2}\Var\Big(\sum\limits^{m}_{i=1,\, i\neq m-1} y_{i}\cdot \big(B^H(s_{i})-B^H(s_{i-1})\big)\Big)\bigg)\, ds\, dy_1\,\cdots dy_{m-2}\, dy_m.
\end{align*}
Repeating the above procedure for all other $y_i$s with $i$ odd gives
\begin{align*}
I^n_{m, m}
&\geq \frac{m!}{(2\pi)^{\frac{3md}{4}}n^{\frac{m}{2}}}\int_{\R^{\frac{md}{2}}}\int_{D_{m,1}} \Big(\prod^{m/2}_{j=1} |\widehat{f}(y_{2j})|^2\Big)\Big(\prod^{m/2}_{j=1} (\Delta s_{2j-1})^{-1}\Big)\\
&\qquad\qquad\qquad\times\exp\bigg(-\frac{1}{2}\Var\Big(\sum\limits^{m/2}_{j=1} y_{2j}\cdot \big(B^H(s_{2j})-B^H(s_{2j-1})\big)\Big)\bigg)\, ds\, d\overline{y},
\end{align*}
where $d\overline{y}=dy_2\, dy_4\,\cdots dy_m$.

Define 
\[
D_{m,2}=\big\{(s_1,s_2,\dots, s_m): n^2\leq\Delta s_{2j-1}\leq \frac{e^{nt}}{m},\; n^{-1}<\Delta s_{2i}<n,\;  j=1,2,\dots,m/2\big\}.
\] 
Note that $D_{m,2}\subseteq D_{m,1}$ when $n$ is large enough. Applying Lemma \ref{cov},
\begin{align} \nonumber \label{inf}
\liminf_{n\to\infty} I^n_{m, m}
&\geq \liminf_{n\to\infty}\frac{m!}{(2\pi)^{\frac{3md}{4}}n^{\frac{m}{2}}}\int_{\R^{\frac{md}{2}}}\int_{D_{m,2}} \Big(\prod^{m/2}_{j=1} |\widehat{f}(y_{2j})|^2\Big)\Big(\prod^{m/2}_{j=1} (\Delta s_{2j-1})^{-1}\Big)\\  \nonumber
&\qquad\qquad\qquad\times\exp\bigg(-\frac{1}{2}\big(1+\frac{c_{10}(1-2H)}{n^{1-2H}}\big)\sum\limits^{m/2}_{j=1} |y_{2j}|^2(\Delta s_{2j})^{2H}\bigg)\, ds\, d\overline{y}\\
& = m!\, t^{\frac{m}{2}}\, \bigg(\frac{1}{(2\pi)^{\frac{d}{2}}}\int^{\infty}_0e^{-\frac{u^{2H}}{2}}\, du\bigg)^{\frac{m}{2}}\bigg(\frac{1}{(2\pi)^d}\int_{\R^{d}} |\widehat{f}(x)|^2|x|^{-d}\, dx\bigg)^{\frac{m}{2}}.
\end{align}

Combining (\ref{sup}) and (\ref{inf}) gives
\begin{align*}
\lim_{n\to\infty} I^n_{m,m}
&=m!\, t^{\frac{m}{2}}\, \bigg(\frac{1}{(2\pi)^{\frac{d}{2}}}\int^{\infty}_0e^{-\frac{u^{2H}}{2}}\, du\bigg)^{\frac{m}{2}}\bigg(\frac{1}{(2\pi)^d}\int_{\R^{d}} |\widehat{f}(x)|^2|x|^{-d}\, dx\bigg)^{\frac{m}{2}}\\
&=C_{f,d}^{m}\, \E\Big(\sqrt{Z(t)}\, \eta\Big)^{m},
\end{align*}
where in the last equality we used Lemma \ref{m} and the following identity 
\[
\frac{2}{(2\pi)^{\frac{d}{2}}}\int^{\infty}_0e^{-\frac{u^{2H}}{2}}\, du=\frac{d}{\pi^{\frac{d}{2}}}\Gamma(\frac{d}{2}).
\]
So the statement  follows. Using the method of moments, the proof is completed.
\end{proof}

\medskip

{\medskip \noindent \textbf{Proof of Theorem \ref{main}.}} Since $f$ is bounded, this follows easily from Proposition \ref{main1}.

\begin{remark}  \label{norm1} Recall the convergence of finite dimensional distributions for two-dimensional Brownian motion in (\ref{kkk}). Theorem \ref{main} implies 
\[
\frac{2}{\pi(2\pi)^{2}} \int_{\R^2} |\widehat{f}(x)|^2 |x|^{-2}\, dx=-\frac{4}{\pi}\int_{\R^2}\int_{\R^2} f(x) f(y)\log|x-y|\, dx\, dy
\]
for all bounded functions $f$ with compact support such that $\int_{\R^2}f(x)\, dx=0$.
\end{remark}

\bigskip

\section{Proof of Theorem \ref{main0}}
In this section, we will show Theorem \ref{main0}, the convergence of finite dimensional distributions. Let $F_n(t)=\frac{1}{\sqrt{n}}\int^{e^{nt}}_0 f(B^H(s))\, ds$. We only need to prove that the finite dimensional distributions of $F_n(t)$ converge to the
corresponding ones of $C_{f,d}\, W(\ell(M^{-1}(t)))$.

Fix a finite number of disjoint intervals $(a_i, b_i]$ with
$i=1,\dots ,N$ and
 $b_i\le a_{i+1}$.  Let $\mathbf{m}=(m_1, \dots, m_N)$ be a fixed multi-index with $m_i\in\N$ for $i=1,\dots ,N$.
 Set $\sum_{i=1}^N m_i=|\mathbf{m}|$ and  $\prod_{i=1}^N m_i!=\mathbf{m}!$.
We need to consider the following sequence of random variables
 \[
        G_n=\prod_{i=1}^N \left( F_n(b_i)- F_n(a_i) \right)^{m_i}
 \]
and compute  $\lim\limits_{n\rightarrow  \infty}  \E(G_n) $. Note that the expectation of $G_n$ can be formulated as 
 \[
      \E(G_n) = \mathbf{m}!\, n^{-\frac{|\mathbf{m}|} 2}  \E \Big(
           \int_{D_\mathbf{m}}  \prod_{i=1}^N\prod_{j=1}^{m_i}f(B^H(s^i_{j}))\, ds \Big),
 \]
where 
\[
D_{\mathbf{m}}=\big\{s \in \R^{|\mathbf{m}|}: e^{na_{i}}<s^i_{1}<\cdots <s^i_{m_i}<e^{nb_{i}}, 1\le i\le N\big\}.
\]

Here and in the sequel we denote the coordinates of a point $s\in\R^{|\mathbf{m}|}$ as
$s= (s^i_j)$, where   $ 1\le i \leq   N$ and $1\leq j \leq m_i $. Define
\begin{equation}
D_{\mathbf{m},1}=\big\{s \in D_{\mathbf{m}}: s^i_j-s^i_{j-1}\geq n^{-|\mathbf{m}|},\; 1\leq j\leq m_i,\; 1\le i\le N\big\}  \label{dm}
\end{equation}
with the convention $s^{i}_0=s^{i-1}_{m_{i-1}}$ for $1\le i\le N$.

For simplicity of notation, we define 
\[
J_0=\big\{(i,j): 1\leq i\leq N, 1\leq j\leq m_i \big\}.
\]
For any $(i_1, j_1)$ and $(i_2,j_2)\in J_0$, we define the following dictionary ordering 
\[
(i_1, j_1)\leq (i_2,j_2)
\]
if $i_1<i_2$ or $i_1=i_2$ and $j_1\leq j_2$. For any $(i,j)$ in $J_0$, under the above ordering, $(i,j)$ is the $(\sum\limits^{i-1}_{k=1}m_k+j)$-th element in $J_0$ and we define $\#(i,j)=\sum\limits^{i-1}_{k=1}m_k+j$.

\begin{proposition} \label{odd} Suppose that at least one of the exponents $m_i$ is odd. Then
\begin{equation*}
\lim\limits_{n\to\infty}\E(G_n)=0.
\end{equation*}
\end{proposition}
\begin{proof}  Let 
\[
M_n=\mathbf{m}!\, n^{-\frac{|\mathbf{m}|} 2}  \E \Big(
           \int_{D_{\mathbf{m},1}}  \prod_{i=1}^N\prod_{j=1}^{m_i}f(B^H(s^i_{j}))\, ds \Big).
\]
It is easy to see
\[
\lim_{n\to\infty}\left[ \E(G_n)-M_n\right]=0.
\]
So it suffices to show  $\lim\limits_{n\to\infty} M_n=0$.

By Fourier transform, $M_n$ is equal to
\begin{align*}
 \frac{\mathbf{m}!}{(2\pi)^{|\mathbf{m}|d}}\,  n^{-\frac{|\mathbf{m}|}{2}}  \int_{\R^{|\mathbf{m}|d}}  \int_{D_{\mathbf{m},1}} 
           \Big( \prod^{N}_{i=1}\prod^{m_i}_{j=1} \widehat{f}(y^i_j)\Big) \exp\bigg(-\frac{1}{2}\Var\Big( \sum^{N}_{i=1}\sum^{m_i}_{j=1} y^i_j\cdot B^H(s^i_j) \Big)\bigg) \, ds\, dy.
\end{align*}
Making the
change of variables $x^i_j=\sum\limits_{(\ell,k)\geq (i,j)} y^{\ell}_k$ for $1\leq i \leq N$ and $1\leq j\leq m_i$,
\begin{align*}
M_n
&=\frac{\mathbf{m}!}{(2\pi)^{|\mathbf{m}|d}}\,  n^{-\frac{|\mathbf{m}|}{2}}  \int_{\R^{|\mathbf{m}|d}}  \int_{D_{\mathbf{m},1}} 
            \prod^{N}_{i=1}\prod^{m_i}_{j=1} \widehat{f}(x^i_j-x^i_{j+1})\\
&\qquad\qquad\times \exp\bigg(-\frac{1}{2}\Var\Big( \sum^{N}_{i=1}\sum^{m_i}_{j=1} x^i_j\cdot \big(B^H(s^i_j)-B^H(s^i_{j-1})\big) \Big)\bigg) \, ds\, dx.
\end{align*}
Applying Proposition \ref{chain},  we obtain
\begin{align*}
\lim_{n\to\infty} M_n
&=\frac{\mathbf{m}!}{(2\pi)^{|\mathbf{m}|d}}\,  \lim_{n\to\infty} n^{-\frac{|\mathbf{m}|}{2}}  \int_{\R^{|\mathbf{m}|d}}  \int_{D_{\mathbf{m},1}} 
           \bigg( \prod_{(i,j)\in J_e} |\widehat{f}(x^i_j)|^2\bigg)\, I_{|\mathbf{m}|}\\
&\qquad\qquad\times \exp\bigg(-\frac{1}{2}\Var\Big( \sum^{N}_{i=1}\sum^{m_i}_{j=1} x^i_j\cdot \big(B^H(s^i_j)-B^H(s^i_{j-1})\big) \Big)\bigg) \, ds\, dx,
\end{align*}
where $J_e=\big\{(i,j)\in J_0: \#(i,j)\; \text{is even}\big\}$ and 
\[
I_{|\mathbf{m}|}=
\begin{cases}
\widehat{f}(x^N_{m_N}), & \text{if}\; |\mathbf{m}| \; \text{is odd};\\
1, & \text{if}\; |\mathbf{m}| \; \text{is even}.
\end{cases}
\]

It is easy to see that $\lim\limits_{n\to\infty} M_n=0$ when $|\mathbf{m}|$ is odd. We shall show $\lim\limits_{n\to\infty}M_n=0$ when $|\mathbf{m}|$ is even. In this case,
\begin{align*}
\lim_{n\to\infty} M_n
&=\frac{\mathbf{m}!}{(2\pi)^{|\mathbf{m}|d}}\,  \lim_{n\to\infty} n^{-\frac{|\mathbf{m}|}{2}}  \int_{\R^{|\mathbf{m}|d}}  \int_{D_{\mathbf{m},1}} 
           \Big( \prod_{(i,j)\in J_e} |\widehat{f}(x^i_j)|^2\Big)\\
&\qquad\qquad\times \exp\bigg(-\frac{1}{2}\Var\Big( \sum^{N}_{i=1}\sum^{m_i}_{j=1} x^i_j\cdot \big(B^H(s^i_j)-B^H(s^i_{j-1})\big) \Big)\bigg) \, ds\, dx.
\end{align*}
Note that the right hand side of the above equality is positive. By Lemma \ref{lnd},
\begin{align*}
 \limsup_{n\to\infty} |M_n|
&\leq c_1  \limsup_{n\to\infty} n^{-\frac{|\mathbf{m}|}{2}}  \int_{\R^{|\mathbf{m}|d}}  \int_{D_{\mathbf{m},1}} 
           \Big( \prod_{(i,j)\in J_e} |\widehat{f}(x^i_j)|^2\Big)\\
&\qquad\qquad\times \exp\bigg(-\frac{\kappa_1}{2}\sum^{N}_{i=1}\sum^{m_i}_{j=1} |x^i_j|^2(s^i_j-s^i_{j-1})^{2H}\bigg) \, ds\, dx\\
&:= c_1 \limsup_{n\to\infty} I_n.
\end{align*}

Assume that $m_{\ell}$ is the first odd exponent. If $s^{\ell}_{m_{\ell}}\geq e^{nb_{\ell}}-nb_{\ell}$, then integrating with respect to proper $x^i_j$s and $s^i_j$s gives
\begin{align*}
I_n 
&\leq c_2\, n^{-1} \sup_{s^{\ell}_{m_{\ell}-1}\in (e^{na_\ell}, e^{nb_{\ell}})} \int_{\R^{d}}  \int_{Q_1}
           |\widehat{f}(x^{\ell+1}_1)|^2(s^{\ell}_{m_{\ell}}-s^{\ell}_{m_{\ell}-1})^{-1}\\
&\qquad\qquad\times   \exp\Big(-\frac{\kappa_1}{2} |x^{\ell+1}_1|^2(s^{\ell+1}_1-s^{\ell}_{m_{\ell}})^{2H}\Big) \, ds^{\ell+1}_1\, ds^{\ell}_{m_{\ell}}\, dx^{\ell+1}_1,
\end{align*}
where $Q_1=\big\{e^{na_{\ell+1}}<s^{\ell+1}_1<e^{nb_{\ell+1}},\, e^{nb_{\ell}}-nb_{\ell}\leq s^{\ell}_{m_{\ell}}<e^{nb_{\ell}},\, s^{\ell}_{m_{\ell}}-s^{\ell}_{m_{\ell-1}}\geq n^{-|\mathbf{m}|}\big\}$.

Integrating with respect to $s^{\ell+1}_1$ and $x^{\ell+1}_1$ gives
\begin{align*}
I_n
&\leq c_3\, n^{-1}  \sup_{s^{\ell}_{m_{\ell}-1}\in (e^{na_\ell}, e^{nb_{\ell}})}  \int^{e^{nb_{\ell}}}_{(s^{\ell}_{m_{\ell}-1}+n^{-|\mathbf{m}|})\vee (e^{nb_{\ell}}-nb_{\ell})}
          (s^{\ell}_{m_{\ell}}-s^{\ell}_{m_{\ell}-1})^{-1}\, ds^{\ell}_{m_{\ell}} \\
&\leq c_4\, n^{-1}\ln(1+n^{|\mathbf{m}|+1}b_{\ell}).
\end{align*}
On the other hand, if $s^{\ell}_{m_{\ell}}\leq e^{nb_{\ell}}-nb_{\ell}$, then integrating with respect to proper $x^i_j$s and $s^i_j$s gives
\begin{align*}
I_n 
&\leq c_5\, n^{-1} \sup_{s^{\ell}_{m_{\ell}-1}\in (e^{na_\ell}, e^{nb_{\ell}})} \int_{\R^{d}}  \int_{Q_2} |\widehat{f}(x^{\ell+1}_1)|^2(s^{\ell}_{m_{\ell}}-s^{\ell}_{m_{\ell}-1})^{-1}\\
&\qquad\qquad\times   \exp\Big(-\frac{\kappa_1}{2} |x^{\ell+1}_1|^2(s^{\ell+1}_1-s^{\ell}_{m_{\ell}})^{2H}\Big) \, ds^{\ell+1}_1\, ds^{\ell}_{m_{\ell}}\, dx^{\ell+1}_1,
\end{align*}
where $Q_2=\big\{e^{na_{\ell+1}}<s^{\ell+1}_1<e^{nb_{\ell+1}},\, s^{\ell}_{m_{\ell-1}}+ n^{-|\mathbf{m}|}\leq s^{\ell}_{m_{\ell}}\leq e^{nb_{\ell}}-nb_{\ell}\big\}$.

Recall $|\widehat{f}(x)|\leq c_{\beta}|x|^{\beta}$ for all $x\in\R^d$.
\begin{align*}
I_n 
&\leq c_6\, n^{-1} \sup_{s^{\ell}_{m_{\ell}-1}\in (e^{na_\ell}, e^{nb_{\ell}})} \int_{\R^{d}}  \int_{Q_2}
           |x^{\ell+1}_1|^{2\beta}(s^{\ell}_{m_{\ell}}-s^{\ell}_{m_{\ell}-1})^{-1}\\
&\qquad\qquad \times \exp\Big(-\frac{\kappa_1}{2} |x^{\ell+1}_1|^2(s^{\ell+1}_1-s^{\ell}_{m_{\ell}})^{2H}\Big) \, ds^{\ell+1}_1\, ds^{\ell}_{m_{\ell}}\, dx^{\ell+1}_1\\
&\leq c_7\, n^{-1}  \sup_{s^{\ell}_{m_{\ell}-1}\in (e^{na_\ell}, e^{nb_{\ell}})}  \int_{Q_2}
           (s^{\ell}_{m_{\ell}}-s^{\ell}_{m_{\ell}-1})^{-1}\, (s^{\ell+1}_1-s^{\ell}_{m_{\ell}})^{-1-2H\beta}\, ds^{\ell+1}_1\, ds^{\ell}_{m_{\ell}}\\
&\leq c_8\,  n^{-1-2H\beta}  \sup_{s^{\ell}_{m_{\ell}-1}\in (e^{na_\ell}, e^{nb_{\ell}})}    \int^{e^{nb_{\ell}}}_{s^{\ell}_{m_{\ell}-1}+n^{-|\mathbf{m}|}} 
           (s^{\ell}_{m_{\ell}}-s^{\ell}_{m_{\ell}-1})^{-1}\, ds^{\ell}_{m_{\ell}}\\
&\leq c_9\, n^{-2H \beta}.
\end{align*}
Therefore, $\lim\limits_{n\to\infty} M_n=0$ and thus $\lim\limits_{n\to\infty}\E (G_n)=0$.
\end{proof}

\bigskip
 Consider now  the convergence of moments when all exponents $m_i$ are even.  
 \begin{proposition} \label{even} Suppose that all exponents $m_i$ are even. Then
 \begin{equation*}  \label{c1}
\lim_{n\to\infty}\E(G_n)=C_{f,d}^{|\mathbf{m}|}\, \mathbb{E} \Big[
\prod_{i=1}^N   \big( W(\ell(M^{-1}(b_i)))- W(\ell(M^{-1}(a_i))) \big)^{m_i}
\Big].
\end{equation*}
\end{proposition}
\begin{proof} Recall the domain $D_{\mathbf{m},1}$ in (\ref{dm}). Define 
\begin{align*}
I^n_{\mathbf{m}}&=\frac{\mathbf{m}!}{(2\pi)^{|\mathbf{m}|d}}\,  n^{-\frac{|\mathbf{m}|}{2}}  \int_{\R^{|\mathbf{m}|d}}  \int_{D_{\mathbf{m},1}} 
            \prod^{N}_{i=1}\prod^{m_i/2}_{k=1} |\widehat{f}(x^i_{2k})|^2\\
&\qquad\qquad\times \exp\bigg(-\frac{1}{2}\Var\Big( \sum^{N}_{i=1}\sum^{m_i}_{j=1} x^i_j\cdot \big(B^H(s^i_j)-B^H(s^i_{j-1})\big) \Big)\bigg) \, ds\, dx.
\end{align*}
The proof of Proposition \ref{chain} says 
\begin{align*}
\lim_{n\to\infty}\big[\E(G_n)-I^n_{\mathbf{m}}\big]=0.
\end{align*}
We make the change of variables $u^i_j=s^i_j-s^i_{j-1}$ for $1\leq j\leq m_i$ and $1\leq i\leq N$. Recall $\theta=(1-2H)\wedge H$. Repeating the procedure for the proof of the inequality (\ref{sup}) gives
\begin{align}  \label{sup1}
\limsup_{n\to\infty} I^n_{\mathbf{m}}
&\leq \frac{\mathbf{m}!}{(2\pi)^{|\mathbf{m}|d}} \limsup_{n\to\infty}\, n^{-\frac{|\mathbf{m}|}{2}}  \int_{\R^{|\mathbf{m}|d}}  \int_{D_{\mathbf{m},1}} 
            \prod^{N}_{i=1}\prod^{m_i/2}_{k=1} |\widehat{f}(x^i_{2k})|^2\\  \nonumber
&\qquad\qquad\times \exp\Big(-\frac{1}{2}\big(1-\frac{c_1(1-2H)}{n^{\theta}}\big)\sum^{N}_{i=1}\sum^{m_i}_{j=1} |x^i_j|^2(s^i_j-s^i_{j-1})^{2H} \Big) \, ds\, dx\\  \nonumber
&= \frac{\mathbf{m}!}{(2\pi)^{\frac{3}{4}|\mathbf{m}|d}} \limsup_{n\to\infty}\, n^{-\frac{|\mathbf{m}|}{2}}  \int_{\R^{|\mathbf{m}|d/2}}  \int_{D_{\mathbf{m},2}} 
            \prod^{N}_{i=1}\prod^{m_i/2}_{k=1}\left(|\widehat{f}(x^i_{2k})|^2(u^i_{2k-1})^{-1}\right)\\  \nonumber
&\qquad\qquad\times \exp\Big(-\frac{1}{2}\big(1-\frac{c_1(1-2H)}{n^{\theta}}\big)\sum^{N}_{i=1}\sum^{m_i/2}_{k=1} |x^i_{2k}|^2(u^i_{2k})^{2H} \Big) \, du\, d\overline{x},
\end{align}
where $d\overline{x}=\prod\limits^{N}_{i=1}\prod\limits^{m_i/2}_{k=1} dx^i_{2k-1}$ and 
\[
D_{\mathbf{m},2}=\Big\{e^{na_i}<\sum_{(\ell,k)\leq (i,j)} u^{\ell}_k<e^{nb_i},\, u^i_j\geq n^{-|\mathbf{m}|},\, 1\leq i\leq N,\, 1\leq j\leq m_i\Big\}.
\]

Define
\begin{align*}
U_{\mathbf{m}}&= \frac{\mathbf{m}!}{(2\pi)^{|\mathbf{m}|d}} \limsup_{n\to\infty}\, n^{-\frac{|\mathbf{m}|}{2}}  \int_{\R^{|\mathbf{m}|d}}  \int_{D_{\mathbf{m},1}} 
            \prod^{N}_{i=1}\prod^{m_i/2}_{k=1} |\widehat{f}(x^i_{2k})|^2\\  \nonumber
&\qquad\qquad\times \exp\Big(-\frac{1}{2}\big(1-\frac{c_1(1-2H)}{n^{\theta}}\big)\sum^{N}_{i=1}\sum^{m_i}_{j=1} |x^i_j|^2(s^i_j-s^i_{j-1})^{2H} \Big) \, ds\, dx.
\end{align*}
We can easily get
\begin{align*}
U_{\mathbf{m}}\leq C^{\frac{|\mathbf{m}|}{2}}_{f,d}\, \mathbf{m}! \limsup_{n\to\infty}\, n^{-\frac{|\mathbf{m}|}{2}}  \int_{O_{\mathbf{m}}} \prod^{N}_{i=1}\prod^{m_i/2}_{k=1} (u^i_{2k-1})^{-1}\, d\overline{u},
\end{align*}
where $d\overline{u}=\prod\limits^{N}_{i=1}\prod\limits^{m_i/2}_{k=1} du^i_{2k-1}$ and
\[
O_{\mathbf{m}}=\Big\{e^{na_i}<\sum_{(\ell,k)\leq (i,j)} u^{\ell}_{2k-1}<e^{nb_i},\; u^i_{2j-1}\geq n^{-|\mathbf{m}|},\; 1\leq i\leq N,\; 1\leq j\leq \frac{m_i}{2}\Big\}.
\]

On the other hand,
\begin{align*}
U_{\mathbf{m}}
&\geq \frac{\mathbf{m}!}{(2\pi)^{|\mathbf{m}|d}} \limsup_{n\to\infty}\, n^{-\frac{|\mathbf{m}|}{2}}  \int_{\R^{|\mathbf{m}|d}}  \int_{D_{\mathbf{m},2}} 
            \prod^{N}_{i=1}\prod^{m_i/2}_{k=1} |\widehat{f}(x^i_{2k})|^2\\  \nonumber
&\qquad\qquad\times \exp\Big(-\frac{1}{2}\sum^{N}_{i=1}\sum^{m_i}_{j=1} |x^i_j|^2(u^i_j)^{2H} \Big) \, du\, dx\\
&\geq \frac{\mathbf{m}!}{(2\pi)^{|\mathbf{m}|d}} \limsup_{n\to\infty}\, n^{-\frac{|\mathbf{m}|}{2}}  \int_{\R^{|\mathbf{m}|d}}  \int_{D_{\mathbf{m},3}} 
            \prod^{N}_{i=1}\prod^{m_i/2}_{k=1} |\widehat{f}(x^i_{2k})|^2\\  \nonumber
&\qquad\qquad\times \exp\Big(-\frac{1}{2}\sum^{N}_{i=1}\sum^{m_i}_{j=1} |x^i_j|^2(u^i_j)^{2H} \Big) \, du\, dx,
\end{align*}
where $D_{\mathbf{m},3}=D_{\mathbf{m},2}\cap \Big\{n^2<u^i_{2k-1},\,1/n<u^i_{2k}<n,\, 1\leq i\leq N,\, 1\leq k\leq m_i/2\Big\}$.

Define
\[
O_{\mathbf{m},1}=O_{\mathbf{m}}\cap \Big\{\sum_{(\ell,k)\leq (i,j)} u^{\ell}_{2k-1}<e^{nb_i}-|\mathbf{m}|n,\; u^i_{2j-1}\geq n^2,\; 1\leq i\leq N,\; 1\leq j\leq \frac{m_i}{2}\Big\}.
\]
We can obtain
\begin{align*}
U_{\mathbf{m}}
&\geq C^{\frac{|\mathbf{m}|}{2}}_{f,d}\, \mathbf{m}! \limsup_{n\to\infty}\,  n^{-\frac{|\mathbf{m}|}{2}}  \int_{O_{\mathbf{m},1}} 
            \prod^{N}_{i=1}\prod^{m_i/2}_{k=1} (u^i_{2k-1})^{-1}\, d\overline{u}\\   \nonumber
&= C^{\frac{|\mathbf{m}|}{2}}_{f,d}\, \mathbf{m}! \limsup_{n\to\infty}\,  n^{-\frac{|\mathbf{m}|}{2}}  \int_{O_{\mathbf{m}}} 
            \prod^{N}_{i=1}\prod^{m_i/2}_{k=1} (u^i_{2k-1})^{-1}\, d\overline{u}.
\end{align*}
Therefore
\begin{equation} \label{supeq}
U_{\mathbf{m}}=C^{\frac{|\mathbf{m}|}{2}}_{f,d}\, \mathbf{m}! \limsup_{n\to\infty}\,  n^{-\frac{|\mathbf{m}|}{2}}  \int_{O_{\mathbf{m}}} 
            \prod^{N}_{i=1}\prod^{m_i/2}_{k=1} (u^i_{2k-1})^{-1}\, d\overline{u}.
\end{equation}

Applying the inequality (\ref{lower}) repeatedly gives
\begin{align*}
\liminf_{n\to\infty} I^n_{\mathbf{m}}
&\geq \frac{\mathbf{m}!}{(2\pi)^{\frac{3}{4}|\mathbf{m}|d}}\liminf_{n\to\infty}\,  n^{-\frac{|\mathbf{m}|}{2}}  \int_{\R^{|\mathbf{m}|d/2}}  \int_{D_{\mathbf{m},1}} 
            \prod^{N}_{i=1}\prod^{m_i/2}_{k=1} \left( |\widehat{f}(y^i_{2k})|^2(s^i_{2k-1}-s^i_{2k-2})^{-1}\right)\\  
&\qquad\qquad\times \exp\bigg(-\frac{1}{2}\Var\Big( \sum^{N}_{i=1}\sum^{m_i/2}_{k=1} y^i_{2k}\cdot \big(B^H(s^i_{2k})-B^H(s^i_{2k-1})\big) \Big)\bigg) \, ds\, dy.
\end{align*}

By Lemma \ref{cov},
\begin{align}   \label{low1}
\liminf_{n\to\infty} I^n_{\mathbf{m}}    
&\geq \frac{\mathbf{m}!}{(2\pi)^{\frac{3}{4}|\mathbf{m}|d}}\liminf_{n\to\infty}\,  n^{-\frac{|\mathbf{m}|}{2}}  \int_{\R^{|\mathbf{m}|d/2}}  \int_{D_{\mathbf{m},2}} 
            \prod^{N}_{i=1}\prod^{m_i/2}_{k=1} \left( |\widehat{f}(y^i_{2k})|^2(u^i_{2k-1})^{-1}\right)\\   \nonumber
&\qquad\qquad\times \exp\bigg(-\frac{1}{2}\big(1+\frac{c_2(1-2H)}{n^{1-2H}}\big)\sum^{N}_{i=1}\sum^{m_i/2}_{k=1} |y^i_{2k}|^2(u^i_{2k})^{2H} \bigg) \, du\, dy\\   \nonumber
&= \frac{\mathbf{m}!}{(2\pi)^{\frac{3}{4}|\mathbf{m}|d}}\liminf_{n\to\infty}\,  n^{-\frac{|\mathbf{m}|}{2}}  \int_{\R^{|\mathbf{m}|d/2}}  \int_{D_{\mathbf{m},3}} 
            \prod^{N}_{i=1}\prod^{m_i/2}_{k=1} \left( |\widehat{f}(y^i_{2k})|^2(u^i_{2k-1})^{-1}\right)\\   \nonumber
&\qquad\qquad\times \exp\bigg(-\frac{1}{2}\big(1+\frac{c_2(1-2H)}{n^{1-2H}}\big)\sum^{N}_{i=1}\sum^{m_i/2}_{k=1} |y^i_{2k}|^2(u^i_{2k})^{2H} \bigg) \, du\, dy\\   \nonumber
&\geq C^{\frac{|\mathbf{m}|}{2}}_{f,d}\, \mathbf{m}! \liminf_{n\to\infty}\,  n^{-\frac{|\mathbf{m}|}{2}}  \int_{O_{\mathbf{m},1}} 
            \prod^{N}_{i=1}\prod^{m_i/2}_{k=1} (u^i_{2k-1})^{-1}\, d\overline{u}\\   \nonumber
&= C^{\frac{|\mathbf{m}|}{2}}_{f,d}\, \mathbf{m}! \liminf_{n\to\infty}\,  n^{-\frac{|\mathbf{m}|}{2}}  \int_{O_{\mathbf{m}}} 
            \prod^{N}_{i=1}\prod^{m_i/2}_{k=1} (u^i_{2k-1})^{-1}\, d\overline{u}.
\end{align}

Let
\begin{align*}
L_{\mathbf{m}}
&=\frac{\mathbf{m}!}{(2\pi)^{\frac{3}{4}|\mathbf{m}|d}}\liminf_{n\to\infty}\,  n^{-\frac{|\mathbf{m}|}{2}}  \int_{\R^{|\mathbf{m}|d/2}}  \int_{D_{\mathbf{m},3}} 
            \prod^{N}_{i=1}\prod^{m_i/2}_{k=1} \left( |\widehat{f}(y^i_{2k})|^2(u^i_{2k-1})^{-1}\right)\\
            &\qquad\qquad\times \exp\bigg(-\frac{1}{2}\big(1+\frac{c_2(1-2H)}{n^{1-2H}}\big)\sum^{N}_{i=1}\sum^{m_i/2}_{k=1} |y^i_{2k}|^2(u^i_{2k})^{2H} \bigg) \, du\, dy.
\end{align*}
Then
\begin{align*}
L_{\mathbf{m}}
&\leq \frac{\mathbf{m}!}{(2\pi)^{\frac{3}{4}|\mathbf{m}|d}}\liminf_{n\to\infty}\,  n^{-\frac{|\mathbf{m}|}{2}}  \int_{\R^{|\mathbf{m}|d/2}}  \int_{D_{\mathbf{m},3}} 
            \prod^{N}_{i=1}\prod^{m_i/2}_{k=1} \left( |\widehat{f}(y^i_{2k})|^2(u^i_{2k-1})^{-1}\right)\\
            &\qquad\qquad\times \exp\bigg(-\frac{1}{2}\sum^{N}_{i=1}\sum^{m_i/2}_{k=1} |y^i_{2k}|^2(u^i_{2k})^{2H} \bigg) \, du\, dy\\     
&= C^{\frac{|\mathbf{m}|}{2}}_{f,d}\, \mathbf{m}! \liminf_{n\to\infty}\, n^{-\frac{|\mathbf{m}|}{2}}  \int_{O_{\mathbf{m},1}} 
            \prod^{N}_{i=1}\prod^{m_i/2}_{k=1} (u^i_{2k-1})^{-1}\, d\overline{u}\\   \nonumber
&= C^{\frac{|\mathbf{m}|}{2}}_{f,d}\, \mathbf{m}! \liminf_{n\to\infty}\,  n^{-\frac{|\mathbf{m}|}{2}}  \int_{O_{\mathbf{m}}} 
            \prod^{N}_{i=1}\prod^{m_i/2}_{k=1} (u^i_{2k-1})^{-1}\, d\overline{u}.       
\end{align*}
Therefore, 
\begin{align}
L_{\mathbf{m}} \label{infeq}
= C^{\frac{|\mathbf{m}|}{2}}_{f,d}\, \mathbf{m}! \liminf_{n\to\infty}\,  n^{-\frac{|\mathbf{m}|}{2}}  \int_{O_{\mathbf{m}}} 
            \prod^{N}_{i=1}\prod^{m_i/2}_{k=1} (u^i_{2k-1})^{-1}\, d\overline{u}.       
\end{align}

When $H=\frac{1}{2}$, inequalities (\ref{sup1}) and (\ref{low1}) become identities. Recall the convergence of finite dimensional distribution for two-dimensional Brownian motion in (\ref{kkk}), Remark \ref{norm1}, (\ref{supeq}) and (\ref{infeq}). We can easily get
\begin{align} \label{lim}
 \mathbf{m}! \lim_{n\to\infty}  n^{-\frac{|\mathbf{m}|}{2}}  \int_{O_{\mathbf{m}}} 
            \prod^{N}_{i=1}\prod^{m_i/2}_{k=1} (u^i_{2k-1})^{-1}\, d\overline{u}=\mathbb{E} \Big[
\prod_{i=1}^N   \big( W(\ell(M^{-1}(b_i)))- W(\ell(M^{-1}(a_i))) \big)^{m_i}\Big].
\end{align}
Combining  (\ref{sup1}), (\ref{supeq}), (\ref{low1}), (\ref{infeq}) and (\ref{lim}) gives the required result.
\end{proof}
 
 \bigskip
 
{\medskip \noindent \textbf{Proof of Theorem \ref{main0}.}}    This follows easily
from Propositions \ref{odd} and \ref{even}.

\bigskip

\begin{minipage}[t]{0.45\textwidth}
{\bf Fangjun Xu}\\ \\
School of Finance and Statistics\\
East China Normal University \\
Shanghai, China 200241\\
\texttt{fangjunxu@gmail.com}
\end{minipage}
\hfill

\end{document}